\documentclass[11pt, a4paper, article,reqno,]{amsart}
\usepackage{graphicx, xcolor,amsmath, amssymb}
\usepackage{geometry}

\def\jan#1{{\textcolor{teal} {#1}}}
\newtheorem{theorem}{Theorem}[section]

\newtheorem{definition}[theorem]{Definition}
\newtheorem{lemma}[theorem]{Lemma}

\theoremstyle{remark}
\newtheorem{remark}[theorem]{Remark}
\DeclareMathOperator{\dive}{div}

\newcommand{\stinf}{Q_\infty}
\newcommand{\stt}{Q_t}
\numberwithin{equation}{section}

\title[Relative entropy for Euler--Navier--Stokes--Korteweg models]{Extending relative entropy for Korteweg-type models with non-monotone pressure: large friction limit and weak-strong uniqueness} 

\author[G. Cianfarani Carnevale]{Giada Cianfarani Carnevale}
\address[Giada Cianfarani Carnevale]{Dipartimento di Ingegneria e Scienze dell'Informazione e Matematica, Universit\`a degli Studi dell'Aquila (Italy)}
\email{giada.cianfaranicarnevale@univaq.it}
\author[J. Giesselmann]{Jan Giesselmann}
\address[Jan Giesselmann]{Technische Universität Darmstadt, Department of Mathematics (Germany)}
\email{giesselmann@mathematik.tu-darmstadt.de}


\allowdisplaybreaks

\begin{document}

\begin{abstract}
In this paper we study weak-strong uniqueness  and singular relaxation limits for the Euler--Korteweg and Navier--Stokes--Korteweg systems with non monotone pressure. Both weak-strong uniqueness and the relaxation limit are investigated using relative entropy technique. We make use of the enlarged formulation of the model in terms of the \emph{drift velocity}  introduced in \cite{Bresh}, generalizing in this way results proved in \cite{LT2} for the Euler--Korteweg model, by allowing more general capillarity functions, and the result contained in \cite{CCL} for the monotone pressure case.

\end{abstract}

%\tableofcontents

\maketitle

\section{Introduction}
\label{sec:intro}

The objective of this paper is to study stability properties of Euler--Korteweg and Navier--Stokes--Korteweg systems with non monotone pressure. In particular, we address weak-strong uniqueness and the high friction limit.
Indeed, we will leverage ideas from \cite{CCL}, which treated monotone pressure laws, in order to extend results from \cite{JT2017} to more general capillarity laws.

The Navier--Stokes--Korteweg (NSK) system  reads
\begin{equation}\label{ek}
	\left\{\begin{aligned}
		& \partial_t \rho + \dive m =0\\
		&  \partial_t m + \dive \left( \frac{m \otimes m}{\rho}\right) + \nabla p(\rho)  =
		 {2 \nu} \dive(\mu_L(\rho)Du) + \nu \nabla(\lambda_L(\rho) \dive u)
		\\
	& \qquad\qquad\qquad	+ \rho \nabla \left( k(\rho)\Delta\rho + \frac{1}{2}k'(\rho)|\nabla \rho|^2 \right) - \xi \rho u,
		\end{aligned}\right.
\end{equation}
where $t>0$, $x\in\mathbb{T}^n$, the $n$--dimensional torus, $\rho$ is the density, $m=\rho u $ is the momentum,  and the constants $\xi > 0$ and $\nu\geq 0$ are the friction and the viscosity coefficient respectively. The Euler--Korteweg (EK) system is obtained by setting $\nu=0$ in \eqref{ek}.
In the viscosity terms of \eqref{ek}
$$Du= \frac{ \nabla u + {}^t\nabla u}{2}$$ 
is the symmetric part of the gradient $\nabla u$ and the Lam\'e coefficients $\mu_L(\rho)$ and $\lambda_L(\rho)$ satisfy the dissipation  relation:
\begin{equation}\label{lame}
    \mu_L(\rho)\geq 0;\ \frac2n \mu_L(\rho) +\lambda_L(\rho) \geq 0.
\end{equation}
By $p$ and $k$ we denote constitutive laws for pressure and capillarity.
The pressure $p(\rho)$ is connected to %  the internal energy $f(\rho)$ and
the enthalpy $h(\rho)$ by the following relations:
\begin{equation}\label{eq:defh}
 % f'(\rho)= \frac{p(\rho)}{\rho^2}, \quad h(\rho)= \rho f(\rho), 
 p(\rho)=\rho h'(\rho) - h(\rho)
 \end{equation}
 which implies
 \begin{equation}\label{eq:defh2}
 h''(\rho) = \frac{p'(\rho)}{\rho}; 
 %\quad 
 %p(\rho)=\rho h'(\rho) - h(\rho);
 \quad \rho \nabla (h'(\rho)) = \rho h''(\rho) \nabla \rho = \nabla p(\rho).   
\end{equation}

We investigate the case of non monotone pressure,  
which is crucial for treating multi-phase flows. For simplicity we assume that the pressure can be split into a power law part and a compactly supported non-convex part, i.e.,  $h(\rho)= \rho^{\gamma}+ e(\rho)$ with $\gamma >1 $ and $e(\rho) \in C_c^{\infty}(0,\infty)$. In view of  the relation $p(\rho)= \rho h' (\rho)-h (\rho)$ in \eqref{eq:defh2} the pressure term satisfies:
\begin{equation}\label{eqp}
p(\rho)= (\gamma-1)\rho^{\gamma}+\rho e'(\rho)- e(\rho)=: p_{\gamma}+p_e.
\end{equation}

We will assume that the capillarity law satisfies 
\begin{equation}
    k(\rho) = \frac{(s+3)^2}{4} \rho^s \quad \text{with } -1 \leq s \leq 2\gamma -3.
\end{equation}
We will prove weak-strong uniqueness in the full parameter range
\begin{equation}
    -1 \leq s \leq 2\gamma -3
\end{equation}
and the large friction limit in the smaller parameter range 
\begin{equation}
    -1 \leq s \leq \gamma -2 .
\end{equation}
%and weak-strong uniqueness under the  assumption

%which is weaker since $\gamma >1$.

These ranges for $s$ should be contrasted with the requirement that $k$ be constant in \cite{JT2017}. In particular, our new condition covers the quantum hydrodynamics case $k(\rho)= \rho^{-1}$.

The entropy of system \eqref{ek}, see \cite{GLT,CCL,JT2017,Bresh}, is defined as:
\begin{equation*}
    \eta (\rho,m,\nabla \rho):= \frac{1}{2}  \frac{|m|^2}{\rho} + h(\rho) + \frac{1}{2}k(\rho) |\nabla \rho|^2
\end{equation*}
that is the sum of the contribution of the kinetic energy, the internal energy ($h(\rho)$) and the potential energy respectively.

%The literature concerning these kind of systems, which include in particular Quantum Hydrodynamic models, is very wide and a complete description of it is beyond the main interest of our present research, which is focused in the study of the relaxation limit for weak, finite energy solutions of \eqref{ek}. In particular, we are not interested here in investigating the existence of such solutions, but solely in understanding their behavior in the high friction regime.  However, for some rigorous   mathematical studies  of such systems, regarding in particular the existence of weak solutions,  the dedicated reader may refer to \cite{AM1,AM2,AS,AS2} and the reference therein. 

The diffusive (large friction and low Mach) limit, after a suitable rescaling of time, in both Navier--Stokes--Korteweg and Euler--Korteweg equations is given by:

\begin{equation}\label{eq:diff-limit}
\rho_t = \dive_x \left( \rho \nabla_x \left(h'(\rho) +  k(\rho)\Delta\rho + \frac{1}{2}k'(\rho)|\nabla \rho|^2 \right) \right).
\end{equation} 
It can be obtained by performing the classical Hilbert expansion, see Section \ref{sec:hilbert}.
Starting from the results contained in \cite{CCL} we consider here only the dissipative model, namely the Navier-Stokes-Korteweg one, with the non-convex pressure as defined in \eqref{eqp}.
It is important to observe that we will assume only the dissipative structure and uniform integrability for the Lam\'e coefficients, without any particular connection with the capillarity terms.\\

The fact that we can allow for a larger range of values $s, \gamma$ than in previous works, in particular \cite{JT2017}, is due the the use of an  
\textit{augmented formulation} of the NSK system \eqref{ek}, that has been used in \cite{Bresh, CCL}. In particular, we need a less stringent assumption on the capillarity coefficient $k(\rho)$ than the one presented in \cite[Lemma 2]{GLT}.
Moreover, in \cite{JT2017} the authors show the weak-strong uniqueness result for Euler-Korteweg with non-monotone pressure only for constant capillarity coefficient. 

Let us introduce the \emph{drift velocity}:
\begin{equation*}
v:= \frac{\nabla \mu(\rho)}{\rho},
\end{equation*}
where $\mu(\rho)$ is related to the capillarity function via $\mu'(\rho) = \sqrt{\rho k(\rho)}$. This allows us to rewrite  \eqref{ek} in augmented form:
\begin{equation}\label{ekb}
	\left\{\begin{aligned}
		& \partial_t \rho + \dive (\rho u) =0\\
		&  \partial_t( \rho u ) + \dive( \rho u\otimes u) + \nabla p( \rho)   = {2 \nu} \dive(\mu_L(\rho)Du) + \nu \nabla(\lambda_L(\rho) \dive u) \\  & \qquad\qquad\qquad +\dive(\mu(\rho)\nabla v) + \frac{1}{2}\nabla(\lambda(\rho)\dive v) - \xi \rho u\\
		& \partial_t (\rho v) + \dive(\rho v \otimes u) + \dive(\mu(\rho)^t\nabla u) + \frac{1}{2}\nabla (\lambda(\rho)\dive u)=0,
		\end{aligned}\right.
\end{equation}
 where
 \begin{align}\label{BDREL}
     \lambda(\rho):= 2(\mu'(\rho)\rho - \mu(\rho)).
 \end{align}
 Now using the Bohm identity (see \cite{Bresh}) we can define a stress tensor $S_1$ in $\eqref{ekb}$ depending only on the capillarity part:
 \begin{equation*}
 \dive(\mu(\rho)\nabla v) + \frac{1}{2}\nabla(\lambda(\rho)\dive v) = \dive S_1
 \end{equation*}
and a new \emph{momentum} $J:=\rho v$ such that we can rewrite the entropy as
\[ \bar \eta(\rho,m,J )= \frac12 \frac{|m|^2}{\rho} + h(\rho) +  \frac12 \frac{|J|^2}{\rho}
\]
and define  a new relative entropy with respect to $(\rho, m, J)$:
\begin{equation}\label{def:rhomJ}
\aligned
\eta(\rho,m,J | \bar{\rho},\bar{m},\bar{J}) &: = \eta(\rho,m,J) - \eta(\bar{\rho},\bar{m},\bar{J}) - \bar{\eta}_{\rho}(\rho- \bar{\rho})- \bar{\eta}_m \cdot (m-\bar{m}) \\
& \  - \bar{\eta}_J \cdot(J-\bar{J}) \\
& = \frac{1}{2} \rho |u - \bar{u}|^2 + \frac{1}{2} \rho |v-\bar{v}|^2 + h_{\gamma}(\rho|\bar{\rho}) + h_{e}(\rho|\bar{\rho}).
\endaligned
\end{equation}
In the forthcoming relative entropy estimates we control part of the previous relation related to non-convexity $h_{e}(\rho|\bar{\rho})$ in terms of the relative energy itself and some positive error terms. We obtain a stability estimate for the convex part of $\eta(\rho,m,J | \bar{\rho},\bar{m},\bar{J})$.

The outline of this work is as follows. In Section \ref{sec:hilbert}, after the appropriate time scaling,  we perform the Hilbert expansion of \eqref{ekb} to identify the limit equation. Then  we rewrite the latter as a correction of the relaxation system \eqref{ekb} in order to perform the relative entropy calculations. Section \ref{sec:relenes} is devoted to obtaining the relative entropy inequality, which will be used to prove a stability estimate between the two solutions in the relaxation limit. 
Finally, in Section \ref{sec:west} we address the issue of weak-strong uniqueness.

\section{Hilbert expansion and formal diffusive limit for the Navier--Stokes--Korteweg model}
\label{sec:hilbert}
In the present part, we briefly recall the correct scaling for which the diffusive limit of \eqref{ekb} can be recovered. More precisely, under the assumption of \eqref{eqp}, for $\xi = 1/\epsilon$, we rescale the time such that  $\partial_t \rightarrow \epsilon\partial_t$ and accordingly, the system \eqref{ekb} reads
\begin{equation}\label{ekb-scaled}
	\left\{\begin{aligned}
		& \partial_t \rho + \frac{1}{\epsilon} \dive (m) =0\\
		&  \partial_t( m) + \frac{1}{\epsilon} \dive \left( \frac{m \otimes m}{\rho}\right) + \frac{1}{\epsilon} \nabla p( \rho)   =
  \frac{1}{\epsilon} \dive{T_\nu}
  %\frac{2 \nu}{\epsilon}\dive(\mu_L(\rho)Du) + \frac{\nu}{\epsilon} \nabla(\lambda_L(\rho) \dive u)
        +\frac{1}{\epsilon}\dive S_1 - \frac{1}{\epsilon^2} \rho u\\
		& \partial_t (J) + \frac{1}{\epsilon} \dive \left( \frac{  J \otimes m}{\rho}   \right) + \dive S_2=0,
		\end{aligned}\right.
\end{equation}
where $J=\rho v$ and   (see \cite{Bresh} for further details)
\begin{align*}
\dive  S_2 & = \dive(\mu(\rho)^t\nabla u) + \frac{1}{2}\nabla (\lambda(\rho)\dive u)\\
\dive T_{\nu} & = 2 \nu \dive(\mu_L(\rho)Du) + \nu \nabla(\lambda_L(\rho) \dive u)
\end{align*}
We introduce the asymptotic expansions of $\rho$ and $m$  in \eqref{ekb-scaled}, the one for $J$ follows from $J = \rho v = \nabla \mu(\rho)$:
\begin{align*}
 &\rho = \rho_0 + \epsilon\rho_1 + \epsilon^2 \rho_2 + \cdots\\
  &m = m_{0} + \epsilon m_{1} + \epsilon^2 m_{2} + \cdots.
  \end{align*}
Then, from the mass conservation we get:
\begin{align*}
&O(\epsilon^{-1}):  & & \dive m_{0} = 0; \\ 
&O(1): & & \partial_t \rho_0 + \dive m_{1} = 0; \\
&O(\epsilon): & &  \partial_t \rho_1 + \dive m_{2} =0;\\
&O(\epsilon^2): & &  \dots\\
\end{align*}
while, from the momentum equation we get:
\begin{align*}
&O(\epsilon^{-2}) :  & &    m_{0}=0; \\ 
&O(\epsilon^{-1}) :  & &    -m_{1} = \nabla p(\rho_0) - \dive S_1(\rho_0); 
\\
&O(1): & &  -m_{2} = \nabla( p'(\rho_0)\rho_1) - \dive(\mu'(\rho_0)\rho_1 \nabla v_0 + \mu(\rho)\nabla v_1) 
\\
& & &   \qquad \qquad + \nabla(\lambda'(\rho_0)\rho_1 \dive v_0 + \lambda(\rho_0) \dive v_1) 
\\
& & & \qquad\qquad - 2 \nu\dive\left(\mu_L(\rho_0)D\left(\frac{m_1}{\rho_0}\right)\right) - \nu \nabla\left(\lambda_L(\rho)\dive \frac{m_1}{\rho_0}\right); \\
&O(\epsilon): & &  \dots\\
\end{align*}

we recover the equilibrium relation $m_{0} = 0$, the Darcy's law $m_{1} = - \nabla_x p(\rho_0) + \dive_x S_1(\rho_0)$ and the gradient flow equation for $\rho_0$, that is the formal limit as $\epsilon \rightarrow 0$ of \eqref{ekb-scaled}:
\begin{equation}\label{gf}
\partial_t\rho_{0} + \dive \left( - \nabla p(\rho_0) + \dive S_1(\rho_0) \right)= 0.
\end{equation}

We use the same strategy as in \cite{LT, LT2,CCL}, i.e. we interpret the latter as the solution of a perturbed Euler-Korteweg system with friction with a perturbation $\bar e$ whose size we can quantify in $\varepsilon$. Let 
$\bar{\rho}$ denote a smooth solution of $\eqref{gf}$. Then $(\bar{\rho}, \bar{m} = \bar{\rho}\bar{u}, \bar{J}= \bar \rho \bar v)$ 
\begin{equation}\label{ekb-scaledstrong}
	\left\{\begin{aligned}
		& \partial_t \bar{\rho} + \frac{1}{\epsilon} \dive \bar{m} =0\\
		&  \partial_t\bar{ m} + \frac{1}{\epsilon} \dive \left( \frac{\bar{m} \otimes \bar{m}}{\bar{\rho}}\right) + \frac{1}{\epsilon} \nabla p(\bar{ \rho})   = \frac{1}{\epsilon}\dive \bar{S_1} - \frac{1}{\epsilon^2} \bar{m} + e(\bar{\rho},\bar{m})\\
		& \partial_t \bar{J} + \frac{1}{\epsilon} \dive \left( \frac{ \bar{J} \otimes \bar{m}}{\bar{\rho}}   \right) + \frac{1}{\epsilon} \dive \bar{S_2}=0.
		\end{aligned}\right.
\end{equation}
where
\begin{equation}\label{eq:defbarutheo}
\bar{m} =  \epsilon \left(- \nabla p(\bar{\rho})+   \dive S_1(\bar{\rho})\right) \qquad  \qquad \bar{J}= \nabla \mu(\bar\rho)
\end{equation}
Clearly, in \eqref{ekb-scaledstrong}, the error term  $e(\bar{\rho},\bar{m})= \bar{e}$ is given by:
\begin{align}\label{eq:error}
 \bar{e} & = \frac{1}{\epsilon} \dive_x \left( \frac{\bar{m} \otimes \bar{m} }{\bar{\rho}} \right) + \bar{m}_t \nonumber\\ 
 & = \epsilon \dive_x \left( \jan{\frac{1}{\bar \rho}} (- \nabla p(\bar{\rho})+   \dive S_1(\bar{\rho})) \otimes  (- \nabla p(\bar{\rho})+   \dive S_1(\bar{\rho}) \right) 
 \nonumber\\ 
 & \ +  \epsilon \left( - \nabla p(\bar{\rho})+   \dive S_1(\bar{\rho}) \right)_t 
 \nonumber\\ 
 & = O(\epsilon).
\end{align}
In view of the previous discussion we notice that the non convexity of the pressure is preserved in the diffusive limit.

\section{Relative entropy estimate for the Navier--Stokes--Korteweg model}
\label{sec:relenes}

Let us start by recalling the  entropy--entropy flux pair $(\eta, Q)$ associated to the original Navier-Stokes-Korteweg system \eqref{ek} with $\xi = 1/\epsilon$ and after the related time scaling. We use the notation of \cite{Bresh, CCL} that we recalled in chapter 1: the \emph{drift velocity}  $v = \sqrt{\frac{k(\rho)}{\rho}} \nabla \rho$, % and its transport equation along the velocity $u$ (see \eqref{ekb-scaled}$_3$). The potential energy rewrites
 % $$F(\rho,J) = h(\rho) + \frac{1}{2} \frac{|J|^2}{\rho},$$
 the drift momentum $J=\rho v$, and the enthalpy $h(\rho)= \rho^{\gamma}+ e(\rho)$. 
 % Hence, the total entropy becomes:
%$$\eta(\rho, m , J) = \frac{1}{2} \frac{|m|^2}{\rho} +h(\rho) + \frac{1}{2} \frac{|J|^2}{\rho},$$
The entropy flux $Q$ is given by:
$$Q(\rho, m, J) = \frac{1}{2} m   \frac{|m|^2}{\rho^2} + m h'(\rho) + \frac{1}{2} m   \frac{|J|^2}{\rho^2} - 2 \nu \mu_L(\rho) Du \cdot \frac{m}{\rho} - \nu \lambda_L(\rho) \dive u \cdot \frac{m}{\rho}.$$
A formal computation gives us:
\begin{multline}\label{single eta}
\partial_t \eta(\rho,m,J) + \frac{1}{\epsilon} \dive Q(\rho,m,J)  = \frac{1}{\epsilon} \frac{m}{\rho} \cdot \dive S_1 - \frac{1}{\epsilon} \frac{J}{\rho} \cdot \dive S_2 - \frac{2 \nu }{\epsilon} \mu_L(\rho) |Du|^2 \\
 - \frac{ \nu }{\epsilon} \lambda_L(\rho) |\dive u|^2  - \frac{1}{\epsilon^2} \frac{|m|^2}{\rho},
\end{multline}
while for the regular solution \eqref{ekb-scaledstrong} we get:
\begin{align}\label{single etabar}
\partial_t {\eta}(\bar \rho,\bar m, \bar J) + \frac{1}{\epsilon} \dive {Q}(\bar \rho,\bar m, \bar J)&  = \frac{1}{\epsilon} \frac{\bar{m}}{\bar{\rho}} \cdot \dive \bar{S_1} - \frac{1}{\epsilon} \frac{\bar{J}}{\bar{\rho}} \cdot \dive \bar{S_2} - \frac{1}{\epsilon^2} \frac{|\bar{m}|^2}{\bar{\rho}}+ \bar{e} \cdot \frac{\bar{m}}{\bar{\rho}}.
\end{align}
where there is no explicit viscous dissipation since it is treated as a solution of perturbed Euler-Korteweg system.

Recall the relative entropy in the new state variables $(\rho,m,J)$, then,
%With the new state variables $(\rho,m,J)$ the relative entropy becomes:
%\begin{align}\label{explicit}
%\eta(\rho,m,J|\bar{\rho},\bar{m},\bar{J}) &=  \eta(\rho,m,J) - \eta(\bar{\rho},\bar{m},\bar{J}) - \eta_{\rho}(\bar{\rho},\bar{m},\bar{J})(\rho- \bar{\rho})  \nonumber
%\\
%&\ - \eta_{m}(\bar{\rho},\bar{m},\bar{J})\cdot(m-\bar{m}) - \eta_{J}(\bar{\rho},\bar{m},\bar{J})\cdot(J-\bar{J}) \nonumber \\
%&= \frac{1}{2} \rho |u - \bar{u}|^2 + \frac{1}{2} \rho |v-\bar{v}|^2 + h(\rho|\bar{\rho})
%\end{align}
%where $h(\rho|\bar \rho)= h_{\gamma}(\rho|\bar \rho)+ e(\rho|\bar \rho)$. 
by assuming that both solutions are regular we get:
\begin{align}\label{formalcomp}
& \partial_t \eta (\rho,m,J| \bar{\rho}, \bar{m}, \bar{J}) + \frac{1}{\epsilon} \dive_x Q(\rho,m,J|\bar{\rho}, \bar{m}, \bar{J}   ) = \\ & - \frac{1}{\epsilon} \rho \nabla \bar{u} : (u-\bar{u}) \otimes (u-\bar{u}) - \frac{1}{\epsilon^2} \rho |u-\bar{u}|^2 - \frac{\rho}{\bar{\rho}} \bar{e}\cdot(u-\bar{u}) - \frac{1}{\epsilon} p(\rho| \bar{\rho})\dive \bar{u} \\ & - \frac{1}{\epsilon} \rho \; \nabla \bar{u} : (v-\bar{v}) \otimes (v-\bar{v})   \\
& \ - \frac{1}{\epsilon}\rho( \mu''(\rho) \nabla \rho - \mu''(\bar{\rho}) \nabla \bar{\rho})) \cdot ((v - \bar{v}) \dive \bar{u}- (u - \bar{u}) \dive \bar{v}) \\ & - \frac{1}{\epsilon} \rho(\mu'(\rho) - \mu'(\bar{\rho}))((v-\bar{v})) \cdot \nabla (\dive\bar{u}) - (u - \bar{u}) \cdot \nabla (\dive \bar{v}))\\
& \ - \frac{2\nu}{\epsilon} \mu_L(\rho)|D(u-\bar u)|^2  - \frac{\nu}{\epsilon}\lambda_L(\rho)|\dive(u-\bar u )|^2 \\
& \ - \frac{2\nu}{\epsilon} \mu_L(\rho) D(\bar u):D(u-\bar u)  - \frac{\nu}{\epsilon} \lambda_L(\rho)\dive \bar u(\dive u - \dive \bar u)
\end{align}
For the formal computation of the relative entropy estimates see \cite{CCL,Bresh}. Now we state the definition of finite energy weak solution for the Navier-Stokes-Korteweg system following the one contained in \cite{LT2,CCL}.
We recall that $p(\rho)= p_{\gamma}+ p_e$, while the capillarity coefficient $k(\rho)$ is given by $k(\rho) = \frac{(s+3)^2}{4} \rho^s$,
for which we obtain $\mu(\rho) = \rho^{\frac{s+3}{2}}$. In the sequel we use the abbreviation $Q_\infty :=(0,+\infty)\times\mathbb{T}^n$ for the space time cylinder.
%with the conditions $\gamma > 1$, $s+2 \leq \gamma$ and $s \geq -1$.

\begin{definition}\label{deFNS}
A tuple ($\rho$, $m$, $J$) with $\rho \in C([0, \infty);(L^1(\mathbb{T}^n))$, $(m,J) \in C([0, \infty);(L^1(\mathbb{T}^n))^{2n})$, $\rho \geq 0$, is a weak (periodic) solution of $\eqref{ekb-scaled}$ if 
\begin{align*}
    & \sqrt{\rho}u,  \sqrt{\rho}v \in L^{\infty}((0,T);L^2(\mathbb{T}^n)^n),\ \rho \in C([0, \infty);(L^\gamma(\mathbb{T}^n)),
    \\
    & \mu_L(\rho)D(u) \in L^1((0,T);L^1(\mathbb{T}^n)^{2n} ),
    \ \lambda_L(\rho)\dive u \in L^1((0,T);L^1(\mathbb{T}^n))
\end{align*}
% \; \frac{\mu(\rho)}{\sqrt{\rho} } \in L^{\infty}((0,T);L^2(\mathbb{T}^n) ) 
and $(\rho, m,J)$ satisfy for all $\psi \in C^1_c([0, \infty); C^1(\mathbb{T}^n))$ and for all $\phi, \varphi \in C^1_c([0, \infty); C^1(\mathbb{T}^n)^n)$: 
\begin{align*}
& - \iint_{\stinf}\Bigg (\rho \psi_t + \frac{1}{\epsilon} m \cdot \nabla_x \psi \Bigg )dxdt = \int_{\mathbb{T}^n} \rho(x,0)\psi(x,0);
\end{align*}
\begin{align*}
&  -  \iint_{\stinf} \Bigg[m \cdot (\phi)_t + \frac{1}{\epsilon}\left(\frac{m \otimes m}{\rho} : \nabla_x \phi \right) + \frac{1}{\epsilon} p(\rho) \dive \phi  -   \frac{2\nu}{\epsilon} \mu_L(\rho) D(u): \nabla \phi \\ 
& \; \; -  \frac{\nu}{\epsilon} \lambda_L(\rho)\dive u \dive \phi 
+ \frac{1}{\epsilon}\left( \mu(\rho) v \cdot \nabla \dive (\phi) + \nabla \mu(\rho) \cdot (\nabla \phi v ) \right) + \\
& \; \; \; \frac{1}{\epsilon} \left( \frac{1}{2} \nabla \lambda(\rho) \cdot v \dive \phi + \frac{1}{2} \lambda(\rho) v \cdot \nabla \dive \phi \right)  \Bigg]dxdt    \\
&= \ - \frac{1}{\epsilon^2} \iint_{\stinf}m \cdot \phi dxdt + \int_{\mathbb{T}^n} m(x,0) \cdot\phi(x,0)dx,
\end{align*}
where we have used the identity 
$$\displaystyle{S= - p(\rho) \mathbb{I} + S_1= - p(\rho)\mathbb{I} + \mu(\rho)\nabla v + \frac{1}{2}\lambda(\rho)\dive v \mathbb{I}},$$ 
and
\begin{align*}
& - \iint_{\stinf} \Bigg[J \cdot\varphi_t + \frac{1}{\epsilon}\left(\frac{J \otimes m}{\rho} : \nabla_x \varphi \right) - \frac{1}{\epsilon}\Bigg( \mu(\rho) u \cdot   (\nabla \dive \varphi ) + \nabla \mu(\rho) \cdot (\nabla \varphi u ) \\
&\ + \frac{1}{2} \nabla \lambda(\rho) \cdot u \dive \varphi + \frac{1}{2}  \lambda(\rho) u \cdot  \nabla \dive \varphi \Bigg)  \Bigg]dxdt =  \int_{\mathbb{T}^n} J(x,0) \cdot \varphi(x,0)dx,
\end{align*}
where we have used the identity 
$$\displaystyle{S_2= \mu(\rho)^t\nabla u + \frac{1}{2}\lambda(\rho)\dive u \mathbb{I}}.$$
If in addition $ \eta (\rho,m,J) \in C([0, \infty); L^1(\mathbb{T}^n))$ and $(\rho,m,J)$ satisfy
\begin{align}\label{dissNS}
& \iint_{\stinf} \left( \eta(\rho,m,J) \right) \dot{\theta}(t) dxdt \leq  \int_{\mathbb{T}^n} \left( \eta(\rho,m,J)\right)|_{t=0} \theta(0)dx \nonumber\\
& \ - \frac{1}{\epsilon^2}  \iint_{\stinf} \frac{|m|^2}{\rho} \theta(t) dxdt - \frac{1}{\epsilon}\iint_{\stinf} \mu_L(\rho)|D(u)|^2\theta(t) dxdt \nonumber \\
& \ - \frac{1}{\epsilon}\iint_{\stinf} \lambda_L(\rho)|\dive u|^2 \theta(t)dxdt
\end{align}
for any non-negative $\theta \in W^{1,\infty}[0, \infty)$ compactly supported on $[0,\infty)$, then $(\rho,m,J)$ is called a \emph{dissipative} weak solution.

If $\eta(\rho,m,J) \in C([0,\infty);L^1(\mathbb{T}^n))$ and $(\rho,m,J)$ satisfy $\eqref{dissNS}$ as an equality, then $(\rho,m,J)$ is called a \emph{conservative} weak solution.

We say that a dissipative (or conservative) weak (periodic) solution $(\rho,m,J)$ of $\eqref{ekb}$ with $\rho \geq 0$ has finite total mass and energy if
$$ \sup_{t \in (0,T)} \int_{\mathbb{T}^n} \rho dx \leq M <+ \infty,$$
and
$$ \sup_{t \in (0,T)} \int_{\mathbb{T}^n}  \eta(\rho,m,J) dx \leq E_o <+ \infty.$$
\end{definition}

In the sequel we use for any $t>0$ the abbreviation $Q_t=(0,t) \times \mathbb{T}^n$.

\begin{theorem}\label{relativeentropyNS}
Let $(\rho,m,J)$ be a dissipative (or conservative) weak solution of $\eqref{ekb-scaled}$ with finite total mass and energy according to Definition \ref{deFNS}, and $\bar \rho$ a smooth solution of $\eqref{ekb-scaledstrong}$. Then
\begin{align}\label{eq:relenNS}
&\int_{\mathbb{T}^n} \eta(\rho,m,J| \bar{\rho}, \bar{m}, \bar{J})(t)dx  \leq  \int_{\mathbb{T}^n} \eta(\rho,m,J| \bar{\rho}, \bar{m}, \bar{J})(0)dx \nonumber \\
&  - \frac{2\nu}{\epsilon}\iint_{ \stt } \mu_L(\rho)|D(u-\bar u)|^2 dxd\tau - \frac{\nu}{\epsilon}\iint_{ \stt }\lambda_L(\rho)|\dive(u-\bar u )|^2dxd\tau  \nonumber\\
& - \frac{1}{\epsilon^2} \iint_{\stt}  \rho |u-\bar{u}|^2 dxd\tau - \frac{1}{\epsilon} \iint_{\stt} \rho \nabla \bar{u}: (u-\bar{u}) \otimes (u-\bar{u})dxdt \nonumber\\ 
& - \iint_{\stt} e(\bar{\rho},\bar{m}) \cdot \frac{\rho}{\bar{\rho}} (u-\bar{u})dxd\tau - \frac{1}{\epsilon} \iint_{\stt}p(\rho|\bar{\rho}) \dive \bar{u} dxd\tau \nonumber \\ 
& - \frac{1}{\epsilon} \iint_{\stt} \rho \; \nabla \bar{u}: (v-\bar{v}) \otimes (v-\bar{v}) dxd\tau \nonumber\\ 
&  - \frac{2\nu}{\epsilon}\iint_{ \stt } \mu_L(\rho) D(\bar u):D(u-\bar u) dxd\tau - \frac{\nu}{\epsilon}\iint_{\stt }\lambda_L(\rho)\dive \bar u(\dive u - \dive \bar u)dxd\tau \nonumber \\
& - \frac{1}{\epsilon} \iint_{\stt} \rho[(\mu''(\rho)\nabla \rho - \mu''(\bar{\rho})\nabla \bar{\rho})\cdot((v-\bar{v})\dive \bar{u} - (u-\bar{u})\dive \bar{v})]dxd\tau \nonumber\\ 
& - \frac{1}{\epsilon}\iint_{\stt} \rho (\mu'(\rho)- \mu'(\bar{\rho}))[(v-\bar{v})\cdot\nabla \dive \bar{u}-(u- \bar{u})\cdot\nabla \dive \bar{v}]dxd\tau,
\end{align}
where 
\begin{equation*}\bar{m} = \bar{\rho}\bar{u}= \epsilon \left(- \nabla p(\bar{\rho})+   \dive S_1(\bar{\rho})\right); \ \bar{J}= \bar{\rho}\bar{v} = \nabla \mu(\bar{\rho}).
\end{equation*}
\end{theorem}
\begin{proof}
The proof is the same as the proof of Theorem 5.2 in \cite{CCL}.% In the forthcoming section we present only the convergence result.
\end{proof}

The non-monotonicity of the pressure will only become relevant in the next section.

\section{Stability result and convergence of the diffusive limit}
\label{sec:stabconv}
With the relative entropy estimate \eqref{eq:relenNS} of Theorem \ref{relativeentropyNS} at hand, we are able to control the convergence to the our diffusive relaxation limit in terms of the convex part of the relative entropy. Indeed, we consider:
\begin{align}\label{eq:distfi}
    \Psi_{\gamma}(t) : & = \int_{\mathbb{T}^n} \left (h(\rho|\bar{\rho}) + \frac{1}{2} \rho \left| \frac{m}{\rho} - \frac{\bar{m}}{\bar{\rho}}\right|^2 + \frac{1}{2} \rho \left|\frac{J}{\rho} - \frac{\bar{J}}{\bar{\rho}}\right|^2 \right )dx - \int_{\mathbb{T}^n} h_e(\rho|\bar \rho)dx \nonumber \\
    & = \int_{\mathbb{T}^n} \left (h_{\gamma}(\rho|\bar{\rho}) + \frac{1}{2} \rho \left| \frac{m}{\rho} - \frac{\bar{m}}{\bar{\rho}}\right|^2 + \frac{1}{2} \rho \left|\frac{J}{\rho} - \frac{\bar{J}}{\bar{\rho}}\right|^2 \right )dx,
\end{align}
with $h(\rho|\bar \rho)= h_\gamma(\rho|\bar \rho)+ h_e(\rho|\bar \rho)$.
The proof of our convergence result will follow the blueprint of \cite{LT,LT2, CCL}. We recall that
\begin{equation*}
   h(\rho) = \frac{\rho^\gamma}{\gamma-1} + e(\rho)\qquad \gamma>1.
 \end{equation*}
A straightforward computation gives us:
\begin{equation}\label{pcontrol}
   p(\rho|\bar{\rho})= (\gamma - 1) h_{\gamma}(\rho|\bar{\rho}) + p_e (\rho|\bar \rho),
\end{equation}
and the term in \eqref{eq:relenNS} involving the pressure will be controlled in terms of the convex part of the relative entropy $\Psi_{\gamma}$. % that is defined in \eqref{eq:distfi} in such a way that we take account only of the convex part. 
The last two integrals of \eqref{eq:relenNS} can be controlled by using the following lemma.
\begin{lemma}\cite[Lemma 35]{Bresh}\label{lemma8}
Let  $\mu(\rho)= \rho^{\frac{s+3}{2}}$ with  $\gamma \geq s+2$ and $s \geq -1$. Then, the following estimate holds $$\rho |\mu'(\rho)-\mu'(\bar{\rho})|^2 \leq C(\bar{\rho})h_{\gamma}(\rho|\bar{\rho}),$$ with $C(\bar{\rho})$ uniformly bounded for $\bar{\rho}$ belonging to compact sets in $\mathbb{R}^+$.   
\end{lemma}

Concerning the non-convex part $e(\rho)$ in \eqref{eq:distfi} we make use of the following result from \cite{JT2017} adapted to our system.
\begin{lemma}\cite[Lemma 3.1]{JT2017}\label{lemmaJAN}
    Let $(\rho,m)$ and $(\bar\rho, \bar m)$ be a weak and strong solution of \eqref{ek}$_1$ respectively then
    \begin{multline}
         \int_{\mathbb{T}^n} e(\rho|\bar \rho)|_tdx  - \int_{\mathbb{T}^n} e(\rho|\bar \rho)|_0 dx = - \frac{1}{\epsilon}\iint_{\stt}  \dive \bar m (e'(\rho)- e'(\bar \rho) - e''(\bar \rho)(\rho-\bar \rho)) dxd\tau \nonumber \\
          + \frac{1}{\epsilon}  \iint_{\stt}  (e''(\rho) \nabla \rho - e''(\bar \rho) \nabla \bar \rho) (m- \bar m) dxd\tau.
    \end{multline}
\end{lemma}
\begin{remark}\label{rem1}
    From the Definition of weak solution \eqref{deFNS} and in particular \eqref{dissNS} we have:
   \begin{align*}
        &\rho^{\frac{s+2}{2}} \in L^2 (\mathbb{T}^n)  \text{ 
 for a.e. t  since } \gamma > s+2, \quad s\geq -1\\
        & \sqrt{\rho}v = C(s) \nabla \rho^{\frac{s+2}{2}} \in L^2 (\mathbb{T}^n) \text{ for a.e. t } 
   \end{align*}
 therefore by Sobolev embedding we get that $\rho^{s+2} \in L^3(\mathbb{T}^n)$ {for a.e. t }
\end{remark}
We are now ready to state our main convergence theorem.

\begin{theorem}\label{theo:stabNSK}
Let $T>0$ be fixed, let $(\rho,m, J)$ be as in Definition $\ref{deFNS}$ and let $\bar{\rho}$ be a smooth solution of $\eqref{ekb-scaledstrong}$ such that $\bar{\rho} \geq \delta > 0$, $\rho>c_p \geq0$, $\bar m$ and $\bar J$ defined as $\eqref{eq:defbarutheo}$. Assume the pressure $p(\rho)$ is given by \eqref{eqp} with $\gamma > 1$. Assume $\mu(\rho) = \rho^{\frac{s+3}{2}}$ with  $-1 \leq s \leq \gamma -2$, and 
	\begin{equation}
	    \label{lamecontrol}
	    \left\| \mu_L(\rho) \right\|_{L^{\infty}((0,t);L^1(\mathbb{T}^n))} , \left\| \lambda_L(\rho) \right\|_{L^{\infty}((0,t);L^1(\mathbb{T}^n))} \leq  \tilde E.
	\end{equation}
	for a positive constant $\tilde E$ independent from $\epsilon$.
	Then, for $t \in [0,T]$, the stability estimate 
	\begin{equation}\label{stabtheosec5}
	\Psi_{\gamma}(t) \leq C(\Psi_{\gamma}(0) + \epsilon^4 +\nu\epsilon),
	\end{equation}
 holds true, where $C$ is a positive constant depending on
	$T$, $M$, the $L^1$ bound for $\rho$,  and $E_o$, the energy bound, both assumed to be uniform in $\epsilon$, $\bar{\rho}$ and its derivatives. Moreover, if $\Psi_{\gamma}(0) \rightarrow 0$ as $\epsilon \rightarrow 0$, then as $\epsilon \rightarrow 0$
	\begin{equation}
	\sup_{t \in [0,T]} \Psi_{\gamma}(t) \rightarrow 0.
	\end{equation}
\end{theorem}
\begin{proof}
In view of the definition of $\Psi_{\gamma}$ in  \eqref{eq:distfi}, from the relative entropy estimate given by   Theorem \ref{relativeentropyNS} we get:
\begin{align}\label{phie}
 &\Psi_{\gamma}(t) + \frac{1}{\epsilon^2} \iint_{\stt} \rho \left| \frac{m}{\rho} - \frac{\bar{m}}{\bar{\rho}} \right|^2dxd\tau  + \frac{2\nu}{\epsilon} \iint_{ \stt } \mu_L(\rho)|D(u)-D(\bar u)|^2 dxd\tau \nonumber \\
 & +\frac{\nu}{\epsilon}\iint_{ \stt } \lambda_L(\rho)[\dive u - \dive \bar u]^2 dxd\tau \\&
 \leq \Psi_{\gamma}(0) + \iint_{\stt} (|Q|+ |E|)dxd\tau 
 + \int_{\mathbb{T}^n}e(\rho|\bar \rho)|_{t} - e(\rho|\bar \rho)|_{0} dx,\nonumber
\end{align}
where the terms $Q$ and $E$ are given by
\begin{equation*}
E : = \bar{e} \cdot \frac{\rho}{\bar{\rho}} \left( \frac{m}{\rho} - \frac{\bar{m}}{\bar{\rho}} \right), \quad Q := Q_1 + Q_2+ Q_3,
\end{equation*}
with
\begin{align*}
Q_1 : = & - \frac{1}{\epsilon} \iint_{\stt} \rho \nabla \bar{u} : [(u-\bar{u}) \otimes (u-\bar{u})]dxd\tau 
\\
& - \frac{1}{\epsilon} \iint_{\stt} \rho \nabla \bar{u} : [(v-\bar{v}) \otimes (v-\bar{v})]dxd\tau 
  - \frac{1}{\epsilon} \iint_{\stt} p(\rho|\bar{\rho}) \dive \bar{u} dxd\tau ,
\end{align*}
\begin{align*}
Q_2 : = & -  \frac{1}{\epsilon} \iint_{\stt} \rho[(\mu''(\rho)\nabla \rho - \mu''(\bar{\rho})\nabla \bar{\rho})\cdot((v-\bar{v})\dive \bar{u} - (u-\bar{u})\dive \bar{v})]dxdt 
\\ 
&  - \frac{1}{\epsilon} \iint_{\stt} \rho (\mu'(\rho)- \mu'(\bar{\rho}))((v-\bar{v}) \cdot \nabla \dive \bar{u} -  (u-\bar{u})\cdot \nabla \dive \bar{v})dxd\tau,
\end{align*}
\begin{align*}
Q_3 &: =  - \frac{2\nu}{\epsilon}\iint_{\stt}\mu_L(\rho) D(\bar u):D(u-\bar u) dxd\tau   - \frac{\nu}{\epsilon}\iint_{ \stt }\lambda_L(\rho)\dive \bar u(\dive u - \dive \bar u)dxd\tau.
\end{align*}
Those terms can be controlled in the same way as in \cite{CCL}. We present here only the final bounds:
\begin{equation*}
\iint_{\stt}|E|dx d\tau \leq CT\epsilon^4 + \frac{1}{8\epsilon^2}\iint_{\stt} \rho\left|\frac{m}{\rho} - \frac{\bar m }{\bar \rho}\right|^2dxd\tau,
\end{equation*}
where $C$  depends on the bound for $\bar \rho$ and on the (uniform)   $L^1$ bound for $\rho$ and $\bar e = O(\epsilon)$.
Concerning the terms $Q_1$ and $Q_2$ we have:
 \begin{equation*}
 \iint_{\stt}|Q_1|+ |Q_2|dxd\tau \leq \frac{1}{8\epsilon^2}\iint_{\stt} \rho\left|\frac{m}{\rho} - \frac{\bar m }{\bar \rho}\right|^2dxd\tau + 
  \tilde C \int_{(0,t)}\Psi_{\gamma}(\tau)d\tau,
 \end{equation*}
where $\tilde  C$  depends on   the bounds $\dive \bar u/ \epsilon=O(1)$ and $\nabla \dive \bar u/ \epsilon=O(1)$. We discuss here only the term involving the pressure:
\begin{align*}
&\frac{1}{\epsilon}\iint_{\stt} p(\rho|\bar{\rho})\dive \bar{u} dxd\tau  = \frac{1}{\epsilon}\iint_{\stt}(\gamma - 1) h(\rho|\bar{\rho})\dive \bar{u} dxd\tau
 + \frac{1}{\epsilon}\iint_{\stt} p_e (\rho|\bar \rho)\dive \bar u dxd\tau \nonumber \\
&\leq C_3 \int_{(0,t)} \Psi_{\gamma}(s)d\tau + C_5 \iint_{\stt} p_e''(\bar{\rho})|\rho-\bar\rho|^2 dxd\tau \nonumber \\
&\leq \hat{c} \int_{(0,t)} \Psi_{\gamma}(s)d\tau
\end{align*}
since the derivatives of $p_e$ are uniformly bounded and $\bar u= O(\epsilon)$.

Finally for $Q_3$:
\begin{multline*}
   \iint_{\stt}|Q_3|dxd\tau  \leq  \nu \bar{C} \epsilon  + \frac{\nu}{\epsilon} \iint_{\stt} \mu_L(\rho)|D(u-\bar u )|^2 dxd\tau \\
    + \frac{\nu}{2\epsilon} \iint_{\stt} \lambda_L(\rho) |\dive (u - \bar u)|^2 dxd\tau .
\end{multline*}

Concerning the last integral in \eqref{phie}:
\begin{equation*}
    \int_{\mathbb{T}^n}e(\rho|\bar \rho)|_{t} - e(\rho|\bar \rho)|_{0} dx
\end{equation*}
we use Lemma \ref{lemmaJAN}, and Theorem 3.2 in \cite{JT2017}. In particular we have to control:
\begin{multline}
         - \frac{1}{\epsilon}\iint_{\stt}  \dive \bar m (e'(\rho)- e'(\bar \rho) - e''(\bar \rho)(\rho-\bar \rho)) dxd\tau 
        \\
        +\frac{1}{\epsilon}\int_{\stt}  (e''(\rho) \nabla \rho - e''(\bar \rho) \nabla \bar \rho) (m- \bar m) dxd\tau \nonumber 
         =: J_1 +J_2.
\end{multline}
Concerning $J_1$, recalling Remark \ref{rem1} and following the approach of Theorem 3.2 \cite{JT2017} and Lemma 1 \cite{GLT} we have:
\begin{multline}
     - \frac{1}{\epsilon}\iint_{\stt} \dive \bar m (e'(\rho)- e'(\bar \rho) - e''(\bar \rho)(\rho-\bar \rho)) dxd\tau  \\
     = - \frac{1}{\epsilon} \iint_{\stt} \dive \bar m e'(\rho|\bar{\rho}) dxd\tau 
      \leq C \iint_{\stt} |\rho-\bar\rho|^2 dxd\tau  
      \leq C \int_0^t \Psi_{\gamma}(\tau) d\tau 
\end{multline}
where $C$ depends on the supremum of the third derivative of $e$.
Concerning $J_2$ we rewrite the integral in the following way since $\rho> c_p\geq 0$:
\begin{align}
    & \frac{1}{\epsilon}\iint_{\stt} (e''(\rho) \nabla \rho - e''(\bar \rho) \nabla \bar \rho) (m- \bar m) dxd\tau \nonumber \\
 & = \frac{1}{\epsilon}\iint_{\stt} \frac{e''(\rho)}{\mu'(\rho)} \rho (v -\bar v) \rho (u- \bar u)dxd\tau + \frac{1}{\epsilon} \iint_{\stt}  \frac{e''(\rho)}{\mu'(\rho)} \rho (v -\bar v) \bar u (\rho- \bar \rho) dxd\tau \nonumber \\
  & + \frac{1}{\epsilon}\iint_{\stt}
 \left( \frac{e''(\rho)}{\mu'(\rho)}\rho - \frac{e''(\bar \rho)}{\mu'(\bar \rho)} \bar \rho \right)  \bar v \rho (u - \bar u) dxd\tau + \frac{1}{\epsilon}\iint_{\stt}
 \left( \frac{e''(\rho)}{\mu'(\rho)}\rho - \frac{e''(\bar \rho)}{\mu'(\bar \rho)} \bar \rho \right)  \bar v \bar u (\rho-\bar \rho) dxd\tau.
\end{align}
Recalling that $e(\rho) \in C_c^{\infty}(0,\infty)$ and using Young's inequality we get:
\begin{align*}
    &\frac{1}{\epsilon}\iint_{\stt} \frac{e''(\rho)}{\mu'(\rho)} \rho (v -\bar v) \rho (u- \bar u)dxd\tau \leq  \frac{1}{8\epsilon^2}\iint_{\stt} \rho|u-\bar u|^2 dxd\tau +         c_1 \iint_{\stt} \rho|v-\bar v|^2 dxd\tau,      \nonumber \\
    &\frac{1}{\epsilon} \iint_{\stt} \frac{e''(\rho)}{\mu'(\rho)} \rho (v -\bar v) \bar u (\rho- \bar \rho) dxd\tau \leq c_2 \iint_{\stt} \rho|v-\bar v|^2 + |\rho-\bar \rho|^2 dxd\tau \leq \int_0^t \Psi_{\gamma}(\tau) d\tau, \nonumber \\
    &\frac{1}{\epsilon} \iint_{\stt}
 \left( \frac{e''(\rho)}{\mu'(\rho)}\rho - \frac{e''(\bar \rho)}{\mu'(\bar \rho)} \bar \rho \right)  \bar v \rho (u - \bar u) dxd\tau \leq \frac{1}{\epsilon} \iint_{\stt} C(\bar \rho,\bar v,M)|\rho-\bar \rho|\sqrt{\rho}|u-\bar u| dxd\tau \nonumber \\
 & \qquad \qquad \qquad \qquad \qquad \qquad \qquad \qquad \qquad \quad \leq \frac{1}{8\epsilon^2}\iint_{\stt} \rho|u-\bar u|^2 dxd\tau + \int_0^t \Psi_{\gamma}(\tau) d\tau ,\\
    & \frac{1}{\epsilon}\iint_{\stt}
 \left( \frac{e''(\rho)}{\mu'(\rho)}\rho - \frac{e''(\bar \rho)}{\mu'(\bar \rho)} \bar \rho \right)  \bar v \bar u (\rho-\bar \rho) dxd\tau 
 \leq  \iint_{\stt} C(\bar \rho,\bar v,M)|\rho-\bar \rho|^2 dx d\tau \leq \int_0^t \Psi_{\gamma}(\tau) d\tau,
\end{align*}
where $c_2$ depends on the maximal density that is in the support of $e$.

Finally the relative entropy inequality becomes:
\begin{align*} 
&\Psi_{\gamma}(t) + \frac{1}{2\epsilon^2} \int \iint_{\stt} \rho \left| \frac{m}{\rho} - \frac{\bar{m}}{\bar{\rho}} \right|^2dxd\tau + \frac{\nu}{\epsilon}\iint_{\stt} \mu_L(\rho)|D(u)-D(\bar u)|^2 dxd\tau \nonumber \\
& + \frac{\nu}{2\epsilon}\iint_{\stt}\lambda_L(\rho)[\dive u - \dive \bar u]^2 dxd\tau \leq \Psi_{\gamma}(0) + \tilde{C}\epsilon^4 + \hat{C} \nu \epsilon + C \int_0^t \Psi_{\gamma}(\tau)d\tau,
\end{align*}
and Gronwall's Lemma concludes the proof.
\end{proof}

\section{Weak-strong uniqueness for Navier-Stokes-Korteweg system with non monotone pressure}\label{sec:west}

In this section, we prove weak-strong uniqueness for NSK under minimal assumptions on $s$ and $\gamma$. It turns out that this can be done, in $n \leq 3$ space dimensions, under significantly weaker assumptions than the proof of convergence to the large friction limit. In the following we always consider $n \leq 3$:
The reason is that in weak-strong uniqueness we deal with one weak solution at a time so we can use the  bound for the $L^2(0,T;H^2(\mathbb{T}^n))$-norm of $\rho^{\frac{s+2}{2}}$. In contrast, in the large friction limit, we would need uniform bounds for this norm for all elements of the sequence which are not available. 
Indeed,
using relative entropy tools and Definition \ref{deFNS}, we show
weak-strong uniqueness for
\begin{equation}\label{nskb}
	\left\{\begin{aligned}
		& \partial_t \rho + \dive (\rho u) =0\\
		&  \partial_t( m ) + \dive\left( \frac{m \otimes m}{\rho}\right) + \nabla p( \rho)   = {2 \nu} \dive(\mu_L(\rho)Du) + \nu \nabla(\lambda_L(\rho) \dive u) + \dive S_1 \\
		& \partial_t (J) + \dive\left( \frac{m \otimes J}{\rho}\right) + \dive S_2=0,
		\end{aligned}\right.
\end{equation}
for $n\leq 3$, $\gamma > \frac{s+3}{2}$, $\lambda_L(\rho)=\lambda(\rho), \mu_L(\rho)=\mu(\rho)$ . 
%\begin{equation}\label{ekb-strong}
%	\left\{\begin{aligned}
%		& \partial_t \bar{\rho} +  \dive \bar{m} =0\\
%		&  \partial_t\bar{ m} +  \dive \left( \frac{\bar{m} \otimes \bar{m}}{\bar{\rho}}\right) +  \nabla p(\bar{ \rho})   =2 \nu\dive(\mu_L(\bar \rho)D\bar u) + \nu \nabla(\lambda_L(\bar \rho) \dive \bar u)  + \dive \bar{S_1} \\
%		& \partial_t \bar{J} +  \dive \left( \frac{ \bar{J} \otimes \bar{m}}{\bar{\rho}}   \right) +  \dive \bar{S_2}=0.
%		\end{aligned}\right.
%\end{equation}

 %   \begin{align*}         & n=3 \text{ that allow us to gain the regularity of $\rho$ in order to impose $\gamma \geq \frac{s+3}{2}$ and not $\gamma > s+2$}\\         & \bar \rho \geq \delta >0 \\&\rho > c_p \geq 0\\         & \lambda_L=\lambda, \; \mu_L=\mu \text{ that allow us to bound $I_2, I_3$ in terms of the relative entropy }\\         & \gamma \geq \frac{s+3}{2}, \; s\geq -1   \end{align*}

By using Lemma \ref{lemmaJAN} and following the structure of Theorem \ref{theo:stabNSK} we prove 
%the convergence result.
\begin{theorem}
Let $T>0$ be fixed, let $(\rho,m, J)$ be a finite energy weak solution of \eqref{nskb} in the sense of Definition \ref{deFNS} such that $\rho>c_p \geq0$.
Let $(\bar{\rho}, \bar m, \bar J)$ be a smooth solution of \eqref{nskb} such that $\bar{\rho} \geq \delta > 0$. 
%$\bar m = \bar \rho \bar u$ and $\bar J= \bar \rho \bar v$.
Assume the pressure $p(\rho)$ is given by \eqref{eqp} with $\gamma > 1$. Assume $\mu(\rho) = \rho^{\frac{s+3}{2}}= \mu_L(\rho)$ and $\lambda(\rho)=\lambda_L(\rho)$ satisfying \eqref{BDREL} with $\gamma \geq \frac{s+3}{2}$ and $s \geq -1$. 
Then, if $(\rho,m, J)$ and $(\bar{\rho}, \bar m, \bar J)$ coincide at $t=0$ then they coincide on all of $[0,T]$.
 \end{theorem}
\begin{proof}
Since we are not performing the high friction limit with $\xi=1/\epsilon$ with the rescaled variable $\epsilon \rightarrow \epsilon t$, as was done in Theorem \ref{relativeentropyNS} and Theorem \ref{theo:stabNSK}, we also have to take care of the dissipative parts. Indeed they are present in the system \eqref{nskb} contrary to \eqref{ekb-scaledstrong}. We follow the strategy of Theorem 5.2 in \cite{CCL} and we obtain:
%\begin{align}\label{pt1}
%& \int_{\mathbb{T}^n} \left( \eta(\rho,m,J) \right)(t) dx \leq  \int_{\mathbb{T}^n} \left( \eta(\rho,m,J)\right)|_{t=0} dx - \iint_{(0,t) \times \mathbb{T}^n} \mu_L(\rho)|D(u)|^2 dxdt \nonumber \\ & \ - \iint_{(0,t) \times \mathbb{T}^n} \lambda_L(\rho)|\dive u|^2 dxdt \end{align} for the strong solution $(\bar \rho, \bar m, \bar J)$: \begin{align}\label{pt2} & \int_{\mathbb{T}^n} \left( \eta( \bar \rho, \bar m, \bar J) \right)(t) dx \leq  \int_{\mathbb{T}^n} \left( \eta( \bar \rho, \bar \bar m,\bar J)\right)|_{t=0} dx - \iint_{(0,t) \times \mathbb{T}^n} \mu_L(\bar \rho)|D(\bar u)|^2 dxdt \nonumber \\ & \ - \iint_{(0,t) \times \mathbb{T}^n} \lambda_L(\bar \rho)|\dive \bar u|^2 dxdt. \end{align} The linear part of the relative entropy estimate  \eqref{explicit} can be obtained by using the Definition \ref{deFNS} and the computation of Theorem 5.2 in \cite{CCL}. We finally get: 
\begin{align}\label{pt3}
& \int_{\mathbb{T}^n} \eta(\rho,m,J| \bar{\rho}, \bar{m}, \bar{J})(t)dx  \leq  \int_{\mathbb{T}^n} \eta(\rho,m,J| \bar{\rho}, \bar{m}, \bar{J})(0)dx  \nonumber \\
& -\iint_{\stt} \rho \nabla \bar{u}: (u-\bar{u}) \otimes (u-\bar{u})dxdt =: J_1 \nonumber \\
& - \iint_{\stt}p(\rho|\bar{\rho}) \dive \bar{u} dxd\tau - \iint_{\stt} \rho \; \nabla \bar{u}: (v-\bar{v}) \otimes (v-\bar{v}) dxd\tau =: J_2\nonumber \\ 
& - \iint_{\stt}\rho[(\mu''(\rho)\nabla \rho - \mu''(\bar{\rho})\nabla \bar{\rho})\cdot((v-\bar{v})\dive \bar{u} - (u-\bar{u})\dive \bar{v})]dxd\tau =: J_3 \nonumber\\ 
& -\iint_{\stt} \rho (\mu'(\rho)- \mu'(\bar{\rho}))[(v-\bar{v})\cdot\nabla \dive \bar{u}-(u- \bar{u})\cdot\nabla \dive \bar{v}]dxd\tau =: J_4 \nonumber \\
& - 2 \nu \iint_{\stt} \mu(\rho)|D(u-\bar u)|^2dxdt - \nu \iint_{\stt} \lambda(\rho)|\dive(u-\bar u)|^2dxdt =: I_1 \nonumber \\
& -2 \nu \iint_{\stt} \rho \left( \frac{\mu(\rho) }{\rho} - \frac{\mu(\bar \rho) }{\bar \rho}\right) D \bar u : D(u-\bar u) dxdt =: I_2 \nonumber \\
& - \nu \iint_{\stt} \rho \left( \frac{\lambda(\rho) }{\rho} - \frac{\lambda(\bar \rho) }{\bar \rho}\right) \dive \bar u \dive(u-\bar u) dxdt =: I_3  \nonumber \\
& 2 \nu \iint_{\stt} \frac{\mu(\bar \rho)}{\bar \rho}\rho\left( \frac{\nabla \rho }{\rho} - \frac{\nabla \bar \rho }{\bar \rho}\right) D \bar u (u-\bar u) dxdt=: I_4 \nonumber \\
& \nu \iint_{\stt} \frac{\lambda(\bar \rho)}{\bar \rho}\rho\left( \frac{\nabla \rho }{\rho} - \frac{\nabla \bar \rho }{\bar \rho}\right) \dive \bar u (u-\bar u) dxdt=: I_5 \nonumber \\
\end{align}
Following the same computations as in Theorem \ref{theo:stabNSK} and Lemma \ref{lemmaJAN} for the pressure term, we get
\begin{align*} 
&J_1 +J_2 \leq  C \int_0^t \Psi_{\gamma}(\tau)d\tau .
\end{align*}
Before providing estimates for $J_3, J_4$ we recall that, since $n\leq 3$ and $\gamma > \frac{s+3}{2}$:
Definition \ref{deFNS} implies $\rho^{\frac{s+2}{2}} \in L^2(0,T; H^2(\mathbb T^n))$ and therefore $\rho^{\frac{s+2}{2}} \in L^2(0,T; C^{0,1/2}(\mathbb T^n))$. In particular $\rho$  is bounded in space, $\rho \in L^2(0,T; C^{0,1/2}(\mathbb T^n))$.
 Another consequence of Definition \ref{deFNS} is continuity in time of $\rho$, i.e. $\rho \in C((0,T);L^1(\mathbb T^n))$.

The following remark will help us in the estimate of the term $J_4$.
\begin{remark}\label{rem_gamma}
Since $\bar \rho$ is bounded from above and below, we have
    \begin{align*}
    %   & \text{If } \rho \in F^{c} : 
    \iint_{\stt}|\rho^{\gamma/2}-\bar \rho^{\gamma/2}|^2dxds %&\leq C_{\bar \rho, \gamma}\iint_{(0,t)\times \mathbb{T}^3}|\rho - \bar \rho|^2dxds \\
        \leq C_{\bar \rho, \gamma}\iint_{\stt} h_{\gamma}(\rho|\bar \rho) dxds.% \\
     %  & \text{If } \rho \in F : \iint_{(0,t)\times \mathbb{T}^3} |\rho^{\gamma/2}-\bar \rho^{\gamma/2}|^2dxds %& \leq C_{\bar \rho, \gamma}\iint_{(0,t)\times \mathbb{T}^3}|\rho - \bar \rho|^{\gamma} dxds \\
     %  \leq C_{\bar \rho, \gamma}\iint_{(0,t)\times \mathbb{T}^3} h_{\gamma}(\rho|\bar \rho) dxds
    \end{align*}
\end{remark}
\begin{proof}
Let us define  $F=\{\rho >0 \, : \,  \rho \leq \frac{\bar \rho}{2} \text{ or } \rho \geq 2\bar \rho\}$.
 We handle the cases $\rho \in F^c$ and $\rho \in F$ separately.  Let $\rho \in F^c$ then:
    \begin{align*}
        \iint_{\stt}|\rho^{\gamma/2}-\bar \rho^{\gamma/2}|^2dxds & \leq \iint_{\stt}\left|\frac{\gamma}{2}\bar \rho^{\frac{\gamma}{2}-1}(\rho-\bar \rho)\right|^2 dxds \leq  C_{\gamma, \bar \rho} \iint_{\stt} |\rho-\bar \rho|^2 dxds \\
        & \leq C_{\gamma, \bar \rho} \iint_{\stt} h_{\gamma}(\rho|\bar \rho) dxds,
    \end{align*}
    where the last inequality comes from Lemma 1 in \cite{GLT}. Let $\rho \in F$ then:
    \begin{align*}
        \iint_{\stt}|\rho^{\gamma/2}-\bar \rho^{\gamma/2}|^2dxds & \leq \iint_{\stt} \left| \rho-\bar \rho\right|^{\frac{\gamma}{2} \cdot 2} dxds \leq  C_{\gamma, \bar \rho} \iint_{\stt} |\rho-\bar \rho|^{\gamma} dxds \\
        & \leq C_{\gamma, \bar \rho} \iint_{\stt} h_{\gamma}(\rho|\bar \rho) dxds
    \end{align*}
where the last inequality comes from Lemma 1 in \cite{GLT}.
\end{proof}

The crucial point in the relative entropy estimate at hand is that by using the regularity of $\rho$ we can relax the assumption on $s$ from $s \leq \gamma -2$  to $s \leq 2\gamma -3 $. %\leq \gamma$ (and not as in Lemma \ref{lemma8}).
Indeed:
    \begin{align}\label{regmu}
      J_4 & \leq C(\bar u)  \iint_{\stt} \rho|\mu'(\rho)-\mu'(\bar \rho)|^2 dxd\tau + \max\{C_2, C_3\}  \iint_{\stt} \rho|u-\bar u|^2 
      %dx d\tau + C_3  \iint_{(0,t)\times \mathbb T^3} 
      +\rho|v-\bar v|^2 dx d\tau \nonumber \\
      & \leq C(\bar u)\iint_{\stt} \rho|\rho^{\frac{s+1}{2}}- \bar \rho^{\frac{s+1}{2}}|^2dx d\tau + \max \left[C_2,C_3 \right] \int_0^t \Psi_{\gamma}(\tau) dx d\tau\nonumber \\
      &\leq C(\bar u,||\rho||_{C^0_t C^{0,1/2}_x}) \iint_{\stt} |(\rho^{\frac{\gamma}{2}})^2 - (\bar\rho^{\frac{\gamma}{2}})^2|^2 dx d\tau + \max \left[C_2,C_3 \right] \int_0^t \Psi_{\gamma}(\tau) d\tau \nonumber \\
    & \leq  C(\bar u,||\rho||_{C^0_t C^{0,1/2}_x})\iint_{\stt} |\rho^{\frac{\gamma}{2}} + \bar \rho^{\frac{\gamma}{2}}|^2  |\rho^{\frac{\gamma}{2}}- \bar \rho^{\frac{\gamma}{2}}|^2 dx d\tau 
    + \max \left[C_2,C_3 \right] \int_0^t \Psi_{\gamma}(\tau) d\tau \nonumber\\
    %&\leq C C_1 \iint_{(0,t)\times \mathbb T^n} |\rho- \bar \rho|^{\gamma} dxd\tau \nonumber      + \max \left[C_2,C_3 \right] \int_0^t\Psi_{\gamma}(\tau) d\tau\nonumber \\
    & \leq C(\bar u, \bar \rho, \gamma, ||\rho||_{C^0_t C^{0,1/2}_x}) \underbrace{\iint_{\stt} h_{\gamma}(\rho|\bar\rho)dxd\tau}_{\text{Thanks to Remark} \; \ref{rem_gamma}}  + \max \left[C_2,C_3 \right] \int_0^t \Psi_{\gamma}(\tau) d\tau  \nonumber \\
    & \leq \tilde C \int_0^t \Psi_{\gamma}(\tau) d\tau
    \end{align}
 where $\tilde C= \max \left[C(\bar u, \bar \rho, \gamma, ||\rho||_{C^0_t C^{0,1/2}_x}), \max [C_2,C_3]\right]$ depends on $||\rho||_{C^0_t C^{0,1/2}_x}$, $||\bar \rho||_{C^0_t C^{0,1/2}_x}$, $s$, $\dive \bar u$, $\dive \bar v$, and $\gamma$.
 
Note that $\frac{\rho \mu''(\rho)\nabla \rho}{\rho} = \frac{(s+1)\mu'(\rho) \nabla \rho}{2 \rho}= \frac{s+1}{2}v$ so that we can bound $J_3$ as follows:
\begin{align}
    J_3 \leq C \int_0^t \Psi_{\gamma}(\tau) d\tau
\end{align}
thanks to $\eqref{regmu}$.

Under the assumption that the Lam\'e  coefficients coincide with the Korteweg ones $\lambda_L=\lambda=\frac{s+1}{2}\mu$, $\mu_L= \mu$ we get:
\begin{align*}
     I_2  & = -2 \nu \frac{2}{s+3} \iint_{\stt} \rho (\mu' (\rho) - \mu'(\bar \rho)) \frac{\mu(\rho)}{\mu(\rho)} D\bar u: D(u-\bar u)  dxd\tau \\
    & \leq  C_1\iint_{\stt} \rho |\mu' (\rho) - \mu'(\bar \rho)|^2 dxd\tau + \nu \iint_{\stt} \mu(\rho)| D(u-\bar u)|^2 dx d\tau \\
    & \leq C_1 \iint_{\stt} h_{\gamma}(\rho|\bar \rho) dxd\tau + \nu \iint_{\stt} \mu(\rho)| D(u-\bar u)|^2 dx d\tau
\end{align*}
where $C_1 = C(s, \bar u, c_p)$.
\begin{align*}
     I_3 & = - \nu \frac{s+1}{s+3} \iint_{\stt} \rho (\mu' (\rho) - \mu'(\bar \rho)) \frac{\lambda(\rho)}{\lambda(\rho)} \dive \bar u: \dive (u-\bar u)  dxd\tau \\
    & \leq C_3 \iint_{\stt} \rho |\mu' (\rho) - \mu'(\bar \rho)|^2 dxd\tau + \frac{\nu}{2} \iint_{\stt} \lambda(\rho)| \dive(u-\bar u)|^2 dxd\tau \\
    & \leq C_3 \iint_{\stt} h_{\gamma}(\rho|\bar \rho) dxd\tau + \frac{\nu}{2} \iint_{\stt} \lambda(\rho)| \dive(u-\bar u)|^2 dxd\tau
\end{align*}
where $C_3= C(s, \bar u, c_p)$.
\begin{align*}
     I_4 & = 2 \nu \max\left(\frac{1}{c_p}, \frac{1}{\delta}\right) \iint_{\stt} \frac{\mu(\bar\rho)}{\bar \rho}\rho (\nabla \rho - \nabla\bar \rho) \dive \bar u (u-\bar u) dxd\tau  \\
    & \leq C_5 \iint_{\stt} \rho |u-\bar u|^2 dxd\tau + C_6 \iint_{\stt} \rho|\nabla \rho - \nabla \bar \rho|^2 dxd\tau  \\
    & = C_5 \iint_{\stt} \rho |u-\bar u|^2 dxd\tau + C_6\iint_{\stt} \rho\left|\frac{\rho}{\mu'(\rho)}v - \frac{\bar \rho}{\mu'(\bar \rho)} \bar v \right|^2 dxd\tau  \\
    & \leq C_5 \iint_{\stt} \rho |u-\bar u|^2 dxd\tau + C_6\max\left(\left|\frac{\rho}{\mu (\rho)}\right|^2, \left|\frac{\bar \rho}{\mu(\bar \rho)}\right|^2\right)\iint_{\stt} \rho\left|v -  \bar v \right|^2 dxd\tau \\
    & \leq C_5 \iint_{\stt} \rho |u-\bar u|^2 dxd\tau + C_6 \iint_{\stt} \rho\left|v -  \bar v \right|^2 dxd\tau
\end{align*}
where $C_5,C_6= C\left(\max\left(\left|\frac{\rho}{\mu (\rho)}\right|^2, \left|\frac{\rho}{\mu(\bar \rho)}\right|^2\right), \nu, c_p, \delta\right)$. We finally get:
\begin{align}
    \Psi_{\gamma}(t) & + \nu \iint_{\stt} \mu(\rho)|D(u-\bar u)|^2dxdt + \frac{\nu}{2} \iint_{\stt} \lambda(\rho)|\dive(u-\bar u)|^2dxdt \nonumber \\
    & \leq \Psi_{\gamma}(0) + C \int_0^t \Psi_{\gamma}(\tau)d\tau.
\end{align}
Applying the Gronwall lemma we conclude the proof.% and obtain \eqref{stabtheosec52}

\end{proof}

\section*{Acknowledgements}
J.G. is grateful for financial support by the Deutsche Forschungsgemeinschaft (DFG) project
“Dissipative solutions for the Navier-Stokes-Korteweg system and their numerical
treatment” 525866748 within SPP 2410 Hyperbolic Balance Laws in Fluid Mechanics: Complexity, Scales, Randomness (CoScaRa).

\end{document}